\documentclass[11pt]{article}

\usepackage{amssymb}
\usepackage{amsmath}
\usepackage{graphics}
\usepackage{color}
\usepackage{hyperref}

\newcommand{\eps}{\varepsilon}

\newtheorem{thm}{Theorem}
\newtheorem{cor}[thm]{Corollary}

\newtheorem{lemma}[thm]{Lemma}
\newtheorem{definition}[thm]{Definition}

\newtheorem{remark}[thm]{Remark}
\newenvironment{proof}{\noindent\bf{Proof.}\rm}{\hfill$\blacksquare$\bigskip}

\newcommand{\p}{\mathbb{P}} 
\renewcommand{\P}{\mathbb{P}} 

\begin{document}

\title{Finding cliques using few probes}

\author{
		Uriel Feige
        \thanks{Weizmann Institute of Science; \texttt{uriel.feige@weizmann.ac.il}. Research supported in part by the Israel Science Foundation (grant No. 1388/16).}
        \and 
        David Gamarnik 
        \thanks{MIT; \texttt{gamarnik@mit.edu}.
        Research supported in part by ONR grant N00014-17-1-2790} 
        \and 
        Joe Neeman 
        \thanks{UT Austin; \texttt{joeneeman@gmail.com}. Research supported in part by the Alfred P. Sloan Foundation.} 
        \and 
        Mikl\'os Z.\ R\'acz 
        \thanks{Princeton University; \texttt{mracz@princeton.edu}. Research supported in part by NSF grant DMS 1811724.}
        \and 
        Prasad Tetali 
        \thanks{Georgia Tech; \texttt{tetali@math.gatech.edu}. Research supported in part by NSF grants DMS-1407657 and DMS-1811935.}
}

\maketitle

\begin{abstract}
Consider algorithms with unbounded computation time that probe the entries of the adjacency matrix of an $n$ vertex graph, and need to output a clique. We show that if the input graph is drawn at random from $G_{n,\frac{1}{2}}$ (and hence is likely to have a clique of size roughly $2\log n$), then for every $\delta < 2$ and constant $\ell$, there is an $\alpha < 2$ (that may depend on $\delta$ and $\ell$) such that no algorithm that makes $n^{\delta}$ probes in $\ell$ rounds is likely (over the choice of the random graph) to output a clique of size larger than $\alpha \log n$. 
\end{abstract}

\section{Introduction}\label{sec:intro}

Consider an algorithm (with unlimited computation time) that may make up to $q = n^{\delta}$ adaptive probes, to be defined later, $1 \le \delta < 2$, to the adjacency matrix of an input graph drawn randomly from $G_{n,\frac{1}{2}}$, and needs to return a clique. What is the largest value of $\alpha$ (as a function of $q$, or of $\delta$) such that the size of the output clique is at least $\alpha \log n$ with probability at least $\frac{1}{2}$ (over the choice of the random graph)? (All logarithms are in base~2.) Observe that necessarily $\alpha \le 2$, because with high probability the size of the largest clique in the input graph is roughly $2\log n$ (see Remark~\ref{rem:matula}).

\begin{remark}
The set of $q$ probes involves at most $2q$ vertices. If $2q < n$, then the algorithm may w.l.o.g.\ remove from the input graph $n - 2q$ vertices prior to making any probes (since the graph is random, it does not matter which vertices are removed), and consequently the assumption that $\delta \ge 1$ is without loss of generality. For $\delta = 2$ the algorithm can probe all entries in the adjacency matrix, and determine the largest clique. Hence we may also assume that $\delta < 2$.
\end{remark}

As a motivation for the types of algorithms considered in this paper, consider the following algorithm for finding a clique in the random graph. Run the greedy algorithm (iteratively choosing an arbitrary vertex and removing all its non-neighbors) until $2^{c\sqrt{\log n}}$ vertices remain, for some choice of constant $c$. Then exhaustively search among these for a clique of size $2c\sqrt{\log n}$. This gives a clique of size $\log n+c\sqrt{\log n}$ in time roughly $n^{2c^2}$.  
It is easy to see that this algorithm inspects $O(n)$ out of $O(n^2)$ total edges.

In the probe model, we allow the algorithms to run in superpolynomial time, but restrict the number of edges the algorithm is allowed to inspect. Specifically, we consider algorithms evolving dynamically over a certain fixed $T$ number of steps. In the first step $t=1$, the algorithm is allowed to choose any pair $(i_1,j_1), 1\le i_1 < j_1 \le n$, and asks the status of this pair, namely whether it is an edge or not. Depending on the outcome, the algorithm selects a second pair $(i_2,j_2), 1\le i_2<j_2\le n$, and asks the status of the pair (edge or not); it then selects $(i_3,j_3)$, and so on. The algorithm runs till $t=T$, may use unbounded computational time, and needs to produce a clique as large as possible. 

If $T=O(n)$, then such an algorithm can produce a clique of size $\frac{3}{2}\log n$ as follows. Run the greedy algorithm until $\sqrt{q}$ vertices remain and thereafter switch to exhaustive search. This gives a clique of size $\log \frac{n}{\sqrt{q}} + 2\log \sqrt{q} = \log n + \tfrac{1}{2} \log q$.  When $q = \Theta(n)$ the algorithm finds a clique of size $\frac{3}{2}\log n$ while inspecting $T=O(n)$ edges. In general, when $q = \Theta(n^{\delta/2})$ and $\delta\ge 1$, the algorithm finds a clique of size $(1 + \delta / 2) \log n$ while inspecting $T=O(n^{\delta})$ edges.

It is an interesting question whether there is any probing type algorithm that finds a clique larger than  $(1+\delta/2)\log n$ with $T=O(n^{\delta})$. Not being able to make progress on this question, we consider the following modification of this question: what happens if the adaptiveness of the algorithm is limited, in the sense that the algorithm is run in a limited number of stages? Specifically, we consider algorithms which can probe the entries of the adjacency matrix in a constant number of rounds, where all probes of a round are done in parallel. Precisely, fix a constant $\ell$. In step $i=1$ the algorithm selects a subset $E_1$ of the set of pairs $(i,j), 1\le i<j\le n$, and the status of all edges in $E_1$ is revealed. Based on the outcome a set $E_2$ of pairs is selected by the algorithm and the status of edges in $E_2$ is revealed. This is continued until at most $\ell$ sets $E_1,\ldots,E_\ell$ are created. The algorithm is limited by having 
\begin{align*}
\sum_{1\le i\le \ell}|E_i|\le T.
\end{align*}
When $T=O(n^{\delta})$, this is effectively a bound $|E_i|=O(n^{\delta})$, since $\ell=O(1)$.

Define $\alpha_{\star} = \alpha_{\star} \left( \delta, \ell \right)$ as the supremum over $\alpha$ such that there exists an algorithm that probes at most $n^{\delta}$ entries of the adjacency matrix of a graph drawn at random from $G_{n,\frac{1}{2}}$, with the probes being done in $\ell$ rounds, and outputs a clique of size at least $\alpha \log n$ with probability at least $1/2$. In this paper we prove both lower and upper bounds on $\alpha_{\star}(\delta, \ell)$. Since the size of the largest clique in a random graph is approximately $2\log n$ (see Remark~\ref{rem:matula}), we immediately have that $\alpha_{\star}\leq 2$. Our main result is an upper bound that is strictly better than this.

\begin{thm}\label{thm:main}
For every $\delta < 2$ and constant $\ell$ we have that $\alpha_{\star} (\delta, \ell) < 2$. 
\end{thm}

We first present algorithms that find cliques in a few rounds of probes ($\ell \in \left\{1,2,3\right\}$) in Section~\ref{sec:algorithms}, providing lower bounds on $\alpha_{\star}(\delta,\ell)$. We then prove Theorem~\ref{thm:main} in Section~\ref{sec:upper_bounds}. In fact, we prove an explicit upper bound as a function of $\delta$ and $\ell$: 
$\alpha_{\star} \left( \delta, \ell \right) \leq 2 - \delta \left( 1 - \delta / 2 \right)^{\ell}$, see Theorem~\ref{thm:main_quantitative}. 

A natural direction for future work is to improve the bounds in this paper. In particular, can one prove an upper bound (smaller than $2$) that holds regardless of the number of rounds $\ell$? Finally, it would be desirable to find the actual value of $\alpha_{\star}(\delta,\ell)$. 

\begin{remark}\label{rem:matula}
The size of the largest clique in a random graph is very well understood. Define 
$\omega_{n} = 2 \log n - 2 \log\log n + 2 \log e - 1$. 
Matula~\cite{matula1972} showed that for any $\varepsilon >0$, the clique number $\omega (G)$ of a random graph $G$ drawn from $G_{n,\frac{1}{2}}$ satisfies $\left\lfloor \omega_{n} - \varepsilon \right\rfloor \leq \omega (G) \leq \left\lfloor \omega_{n} + \varepsilon \right\rfloor$ with probability tending to $1$ as $n \to \infty$; see also~\cite{bollobas1976cliques}. To simplify exposition, in the following we will often neglect lower order terms and just state that the size of the largest clique is approximately $2 \log n$ with high probability (w.h.p.). In all such cases the arguments can be made precise by applying Matula's result; we leave the details to the reader. 
\end{remark}

\subsection{Related work}

The problem of finding structure in a random graph by adaptively querying the existence of edges between pairs of vertices was recently introduced by Ferber, Krivelevich, Sudakov, and Vieira~\cite{ferber2016finding,ferber2017finding}. In particular, they studied finding a Hamilton cycle~\cite{ferber2016finding} and finding long paths~\cite{ferber2017finding} in the adaptive query model. Conlon, Fox, Grinshpun, and He~\cite{CFGH18} also study this problem, which they term the Subgraph Query Problem, focusing on finding a copy of a target graph $H$, and in particular studying the case when $H$ is a small (constant size) clique. All of these works focus on sparse random graphs. 

In a similar vein, we study finding large cliques in an adaptive query model. The main difference between our work and those mentioned above---apart from the underlying random graph being dense or sparse---is that we study a query model where adaptiveness is limited to constantly many rounds of probes. It would be interesting to understand the effect of limited adaptiveness on finding other types of structures, such as a Hamilton cycle, long paths, or a particular target graph.

\section{Algorithms}\label{sec:algorithms} 

In this section we present algorithms that find cliques with one, two, and three rounds of probes. Interestingly, the algorithm for $\ell = 3$ rounds of probes described in Section~\ref{sec:three_rounds} matches (to leading order) the performance of the greedy algorithm presented in Section~\ref{sec:intro}, which corresponds to the {\em adaptive} probe model with no restriction on $\ell$ ($\ell = \infty$).  

\subsection{One round algorithm}

\begin{lemma}\label{lem:one_round}
For every $1\leq \delta < 2$ we have that $\alpha_{\star}(\delta,1) = \delta$. 
\end{lemma}
\begin{proof}
Let $q=n^\delta$.
A simple one round algorithm probes all entries induced by $\sqrt{2q}$ vertices. This finds a clique of size approximately $2\log \sqrt{2q} \simeq \log q$ (w.h.p.). 

This is optimal up to lower order terms, due to the fact that in a complete graph, the pattern of $q$ probes that maximizes the number of complete subgraphs of any given size is a clique of size $\sqrt{2q}$. (See Section~\ref{sec:upper_bounds} for details on the upper bound.)
\end{proof}

\subsection{Two rounds}

\begin{lemma}\label{lem:two_rounds}
For every $1\leq \delta \le \frac{6}{5}$ we have that $\alpha_{\star}(\delta,2) \geq \frac{4\delta}{3}$. 
For every $\frac{6}{5} \le \delta  < 2$ we have that $\alpha_{\star}(\delta,2) \geq 1 + \frac{\delta}{2}$. 
\end{lemma}
\begin{proof}
We first present the algorithm for the case when $q = \Theta(n)$.

{\bf Round 1}: Probe all edges induced by a set $S\subset V$ of size $n^{1/6}$, and all edges connecting $S$ to $T\subset V$ of size $n^{5/6}$ (where $S$ and $T$ are disjoint).

{\bf Round 2}: Let $S' \subset S$ be a clique of largest size within $S$, and let $T' \subset T$ be the set of vertices within $T$ that are neighbors of all vertices of $S'$. Note that $S'$ has size approximately $\frac{1}{3}\log n$ and $T'$ has size approximately $\sqrt{n}$ (w.h.p.). Now probe all edges in $T'$, finding a clique of size approximately $\log n$ (w.h.p.). Together with $S'$, this gives a clique of size approximately $\frac{4}{3} \log n$ (w.h.p.).

For $q = \Theta(n^{\delta})$ the same algorithm applies, with different set sizes. For $1\leq \delta \le \frac{6}{5}$, choose $S$ to have size $n^{\delta/6}$ and $T$ to have size $n^{\delta} / |S| = n^{5\delta/6}$. Then $S'$ has size approximately $\frac{\delta}{3}\log n$, while $T'$ has size approximately $n^{\delta/2}$ (w.h.p.). The largest clique in $T'$, together with $S'$, gives a clique of size approximately $\frac{4\delta}{3} \log n$ (w.h.p.).
For $\frac{6}{5} \le \delta  < 2$, choose $S$ to have size $n^{\frac{1}{2} - \frac{\delta}{4}}$ and $T$ to have size $n \le n^{\delta} / |S|$. Then $S'$ has size approximately $(1 - \frac{\delta}{2})\log n$, while $T'$ has size approximately $n^{\delta/2}$ (w.h.p.). The largest clique in $T'$, together with $S'$, gives a clique of size approximately $(1 + \frac{\delta}{2}) \log n$ (w.h.p.).
\end{proof}

\subsection{Three rounds}\label{sec:three_rounds}

\begin{lemma}\label{lem:three_rounds}
For every $1\leq \delta < 2$ we have that $\alpha_{\star}(\delta,3) \geq 1 + \delta/2$. 
\end{lemma}
\begin{proof}
We first present the algorithm for the case when $q = \Theta(n)$.

{\bf Round 1}: Probe all edges induced by a set $S\subset V$ of size $n^{1/4}$.

{\bf Round 2}: Let $S' \subset S$ be a clique of largest size within $S$, and note that $S'$ has size approximately $\frac{1}{2}\log n - 2 \log\log n$ (up to an additive constant, w.h.p.). Let $T$ be a set of $\frac{n}{\log n}$ vertices, with $S$ and $T$ disjoint. Probe all edges between $S'$ and $T$.

{\bf Round 3}: Let $T' \subset T$ be a set of $\sqrt{n}$ vertices that are neighbors of all vertices of $S'$. Such a set $T'$ exists with high probability. Probe all edges in $T'$, finding a clique of size approximately $\log n$ (w.h.p.). Together with $S'$, this gives a clique of size approximately $\frac{3}{2}\log n$ (w.h.p.).

For $q = \Theta(n^{\delta})$ the same algorithm applies, with different set sizes. Choose $S$ to have size $n^{(1-\delta/2)/2}$, in which case $S'$ has size $(1-\delta/2) \log n - 2 \log \log n$ (up to an additive constant, w.h.p.). The set $T'$ will now have size $n^{\delta/2}$, and altogether we obtain a clique of size approximately $(1+\delta/2) \log n$ (w.h.p.). 
\end{proof}

\section{Upper bounds}\label{sec:upper_bounds}

In proving upper bounds on the size of the clique that can be found, we use the following definition.

\begin{definition}
\label{def:cover}
Given a graph $G=(V,E)$, a set $S \subset V$, and a parameter $0 \le \beta \le 1$, we say that $S$ is $\beta$-covered if the subgraph induced by $S$ contains at least $\beta\binom{|S|}{2}$ edges. Given positive integers $n$, $m$, and $k$ (with $k < n$), and a parameter $0 \le \beta \le 1$, we let $N_{n,m,k,\beta}$ denote the maximum number of sets of size $k$ that can be $\beta$-covered in an $n$ vertex graph with $m$ edges.
\end{definition}

The following theorem gives an upper bound on $N_{n,m,k,\beta}$. 

\begin{thm}
\label{thm:cover}
Let $G=(V,E)$ be a graph with $n$ vertices and $m$ edges. Then the number $N_{n,m,k,\beta}$ of sets $S \subset V$ of size $k$ that are $\beta$-covered satisfies:

\begin{itemize}

\item $N_{n,m,k,\beta} \le m^{(1 - \sqrt{1 - \beta})k + 1}n^{(2\sqrt{1 - \beta} - 1)k + 2}$ when $\beta \le \frac{16}{25}$, and 

\item $N_{n,m,k,\beta} \le m^{\sqrt{\beta}k/2 + 1}n^{(1 - \sqrt{\beta})k + 2}$ when $\beta \ge \frac{16}{25}$.

\end{itemize}

Moreover, these upper bounds are tight up to lower order multiplicative terms (when $k$ is much smaller than $m$ and $n$).
\end{thm}

We defer the proof of Theorem~\ref{thm:cover} to Section~\ref{sec:cover}. Using Theorem~\ref{thm:cover} we are now ready to prove Theorem~\ref{thm:main}. In fact, we state and prove a quantitative bound on $\alpha_{\star} (\delta, \ell)$, which implies Theorem~\ref{thm:main}. For $0 \leq \beta \leq 1$ and $1 \leq \delta < 2$ define   
\begin{equation}\label{eq:f}
f(\beta,\delta) :=
\begin{cases}
  (2 - \delta)\sqrt{1-\beta} + \delta - 1, &\text{ if } \beta \in \left[ 0, \frac{16}{25} \right], \\
  1 - \left(1 - \frac{\delta}{2} \right) \sqrt{\beta}, &\text{ if } \beta \in \left[ \frac{16}{25}, 1 \right].
\end{cases}
\end{equation}
\begin{thm}\label{thm:main_quantitative} 
We have that 
\begin{equation}\label{eq:upper_bound}
\alpha_{\star} \left( \delta, \ell \right) 
\leq 
\min_{\beta \in \Delta_{\ell-1}} \max_{i \in \left\{ 1, \ldots, \ell \right\}} \frac{2 f\left( \sum_{j=1}^{i} \beta_{j}, \delta \right)}{\sum_{j=i}^{\ell} \beta_{j}},
\end{equation}
where $f$ is defined in~\eqref{eq:f} and $\Delta_{\ell-1}$ denotes the $(\ell-1)$-simplex. 

In particular, there exists a choice of $\beta \in \Delta_{\ell-1}$ in the formula above such that for every $1 \leq \delta < 2$ and constant $\ell$ we have that 
\begin{equation}\label{eq:explicit_bound}
\alpha_{\star} (\delta, \ell) 
\leq 2 - \delta \left( \frac{2-\delta}{2} \right)^{\ell}.
\end{equation}
\end{thm}

\begin{proof} 
Let $A$ be a deterministic (w.l.o.g., because the input is randomized) algorithm that takes $\ell$ rounds, for some constant $\ell$, and makes $q = n^{\delta}$ probes in total, where $1 \le \delta < 2$. We set $k = \alpha\log n$, for some $1 < \alpha < 2$, to be determined later as a function of $\ell$ and $\delta$. We will show that $A$ fails (w.h.p.) to find cliques of size $k$.

Let us fix nonnegative $\beta_1, \ldots, \beta_{\ell}$ satisfying $\sum_{i=1}^{\ell} \beta_i = 1$, whose values will later be optimized as a function of $\delta$. Consider $\ell$ identical copies of $A$. On a given input, all copies of $A$ run in an identical fashion and all tentatively produce the same output clique $K$ (of size $k$).
We say that round $i$ is {\em significant} if the number of probes to $K$ in rounds~1 up to~$i-1$ is at most $\sum_{j=1}^{i-1} \beta_j \binom{k}{2}$ and the number of probes to $K$ in rounds~1 up to~$i$ is at least $\sum_{j=1}^{i} \beta_j \binom{k}{2}$. Given $K$ and the sequence of probes, there must be at least one significant round; this can be proven by induction for instance.  Copy $i$ of $A$ outputs $K$ if the first significant round is $i$, and outputs nothing otherwise. This view of multiple copies of $A$ is identical (in its output) to a single copy of $A$. We shall show that each copy succeeds to output a clique of size $k$ with probability less than $1/(2\ell)$, and hence by a union bound the algorithm fails with probability at least $1/2$.

Consider now a single copy of $A$, say $A_i$. A set $S \subset V$ of size $k$ is referred to as an $i$-eligible set if the number of probes to $S$ in rounds~1 up to~$i-1$ is at most $\sum_{j=1}^{i-1} \beta_j \binom{k}{2}$ and the number of probes to $S$ in rounds~1 up to~$i$ is at least $\sum_{j=1}^{i} \beta_j \binom{k}{2}$. By definition, $A_{i}$ is only allowed to output an $i$-eligible set. Note that to determine whether a set $S$ is $i$-eligible, it suffices to see the answers to all probes up to round $i-1$, as this determines the sets of probes also in round $i$ (and $i$-eligibility does not depend on the answers to the probes in round $i$).
Let $E_i$ be the event that at least one of the $i$-eligible sets is indeed a clique. Note that algorithm $A_{i}$ produces no output unless event $E_{i}$ holds, and hence the probability that $A_{i}$ succeeds is bounded above by the probability of $E_{i}$. 
For each $i$-eligible set, after round $i-1$ there are at least $\sum_{j=i}^{\ell} \beta_{j} \binom{k}{2}$ pairs of vertices that have not yet been probed and hence the probability that this set is a clique is at most $2^{-\sum_{j=i}^{\ell} \beta_{j} \binom{k}{2}}$. 
To upper bound the number of $i$-eligible sets, observe that at least $\sum_{j=1}^{i} \beta_{j} \binom{k}{2}$ pairs of vertices of a given $i$-eligible set are probed up to round $i$ (we do not care how these probes are distributed among rounds $1$ to $i$). Therefore the number of $i$-eligible sets is at most $N_{n,q,k,\sum_{j=1}^{i} \beta_{j}}$ (see Definition~\ref{def:cover}). A union bound thus gives us that 
\begin{equation}\label{eq:union_bd}
\P \left( E_{i} \right) \leq N_{n,q,k,\sum_{j=1}^{i} \beta_{j}} 2^{-\sum_{j=i}^{\ell} \beta_{j} \binom{k}{2}}.
\end{equation}
By Theorem~\ref{thm:cover} we have that 
\[
\log N_{n,n^{\delta},k,\beta} \leq \alpha f\left( \beta, \delta \right) \log^{2} n + 4 \log n.
\]
Consequently, using the notation $s_{i} := \sum_{j=1}^{i} \beta_{j}$ and taking logarithms in~\eqref{eq:union_bd}, we obtain that 
\begin{align*}
\log \P \left( E_{i} \right) 
&\leq 
\alpha f \left( s_{i}, \delta \right) \log^{2} n + 4 \log n - \frac{1 - s_{i-1}}{2} \left( \alpha^{2} \log^{2} n - \alpha \log n \right) \\
&= \alpha \left\{ f \left( s_{i}, \delta \right) - \frac{1-s_{i-1}}{2} \alpha \right\} \log^{2} n + \left( \frac{1-s_{i-1}}{2} \alpha + 4 \right) \log n.
\end{align*} 
We thus see that if 
$\alpha > 2 f\left( s_{i}, \delta \right) / (1 - s_{i-1})$ 
then $\p \left( E_{i} \right) \to 0$ as $n \to \infty$. 
If  this holds for every $i \in \left\{1, \ldots, \ell \right\}$, this proves~\eqref{eq:upper_bound}. 

For $\ell = 1$, the expression~\eqref{eq:upper_bound} gives $\alpha_{\star} \left( \delta, 1 \right) \leq \delta$, which is tight (see Lemma~\ref{lem:one_round}). In the following we assume that $\ell \geq 2$. 
For constant $\ell \geq 2$ and fixed $1 \leq \delta < 2$, the expression in~\eqref{eq:upper_bound} gives an optimization problem in $\beta \in \Delta_{\ell-1}$ to solve to obtain an explicit upper bound on $\alpha_{\star}(\delta, \ell)$. This is pursued in more detail in Section~\ref{sec:explicit}; in particular, the optimum is found when $\delta = 1$. Here we simply choose a particular $\beta' \in \Delta_{\ell-1}$ that implies~\eqref{eq:explicit_bound} (and hence also proves Theorem~\ref{thm:main}). 

Specifically, let 
\[ 
\beta_{i}' := r^{i-1} \eps
\]
for $i = 1, \ldots, \ell$, 
where 
\[
r := \frac{2}{2-\delta}
\]
and, since $\sum_{i=1}^{\ell} \beta_{i}' = 1$, we have that $\eps = \left( r - 1 \right) / \left( r^{\ell} - 1 \right)$. Define also, as above, $s_{i}' := \sum_{j=1}^{i} \beta_{j}'$. By~\eqref{eq:upper_bound}, in order to show~\eqref{eq:explicit_bound}, it suffices to show that 
\begin{equation}\label{eq:to_show_beta}
\frac{2 f \left( s_{i}', \delta \right)}{1 - s_{i-1}'} \leq 2 - \delta r^{-\ell}
\end{equation}
for every $i \in \left\{ 1, \ldots, \ell \right\}$.

First, for $i = \ell$ we have that $s_{\ell}' = 1$ and hence $f\left( s_{\ell}',\delta \right) = \delta / 2$. We also have that 
$1 - s_{\ell-1}' = \left( r^{\ell} - r^{\ell-1} \right) / \left( r^{\ell} - 1 \right)$. Hence 
\[
\frac{2 f \left( s_{\ell}', \delta \right)}{1 - s_{\ell-1}'} 
= \delta \frac{r^{\ell}-1}{r^{\ell}-r^{\ell-1}} 
= 2 - \frac{\delta}{r^{\ell}-r^{\ell-1}},
\]
where the second equality follows from the definition of $r$. Now~\eqref{eq:to_show_beta} follows by dropping the $r^{\ell-1}$ term in the right hand side of the display above. 

We now turn to $i \leq \ell - 1$. Note that $s_{i}' \leq s_{\ell-1}' = \left( r^{\ell-1} - 1 \right) / \left( r^{\ell} - 1 \right) < 1/r = (2-\delta)/2 \leq 1/2$ and hence by the definition of $f$ (see~\eqref{eq:f}) we have that 
\[
\frac{2 f \left( s_{i}', \delta \right)}{1 - s_{i-1}'} 
= \frac{2 \left\{ (2-\delta) \sqrt{1-s_{i}'} + \delta - 1 \right\}}{1 - s_{i-1}'}.
\]
Using the bound $\sqrt{1-s_{i}'}\leq 1 - s_{i}'/2$ we obtain that 
\[
\frac{2 f \left( s_{i}', \delta \right)}{1 - s_{i-1}'} 
\leq 
2 - \frac{\left( 2 - \delta \right) s_{i}' - 2 s_{i-1}'}{1 - s_{i-1}'} 
= 2 - \frac{\delta}{r^{\ell} - r^{i - 1}}.
\]
Now~\eqref{eq:to_show_beta} follows by dropping the $r^{i-1}$ term in the display above. 
\end{proof}

\subsection{Bounding $N_{n,m,k,\beta}$: an extremal problem}\label{sec:cover}

In this subsection we prove Theorem~\ref{thm:cover}, which gives essentially the tight upper bound on $N_{n,m,k,\beta}$ (see Definition~\ref{def:cover}), as we will show.

Given integers $k$ and $1 \le t \le \binom{k}{2}$, let $M(k,t)$ denote the minimum, over all $k$-vertex graphs $H$ with $t$ edges, of the size (number of edges) of the maximum matching in $H$. We use the notation $M(k,\beta)$ when $t$ is expressed as $\beta\binom{k}{2}$.

\begin{lemma}
\label{lem:covered}
Using the notation as above, 
$N_{n,m,k,\beta} \le \binom{m}{M(k,\beta)} \binom{n}{k - 2M(k, \beta)}$.
\end{lemma}

\begin{proof}
For a set $S$ of size $k$ to be $\beta$-covered, its induced subgraph must have a matching of size $M(k, \beta)$. We encode $S$ by its matching edges followed by the remaining vertices in $S$. Given that $G$ has $m$ edges, there are at most $\binom{m}{M(k, \beta)}$ ways of encoding the matching edges, and then at most $\binom{n}{k - 2M(k, \beta)}$ ways of choosing the remaining vertices.
\end{proof}

\begin{lemma}
\label{lem:matching}
Define 
\begin{equation}\label{eq:mu}
\mu(k, \beta) := \min\left\{\frac{\sqrt{\beta}}{2}, \left( 1 - \sqrt{1 - \beta} \right)\right\} k.
\end{equation}
We then have that 
\[
\left\lfloor \mu(k,\beta) \right\rfloor 
\leq M(k,\beta) 
\leq \left\lfloor \mu(k,\beta) \right\rfloor + 1.
\]
\end{lemma}

\begin{proof}
Let $H$ be an arbitrary graph on $k$ vertices. It is well known (see, e.g.,~\cite{LP2009matching}) that if the maximum matching in $H$ has size $M$, then $H$ has a Gallai-Edmonds (GE) decomposition satisfying $c - s = k - 2M$, where $c$ is the number of odd components $C_1, \ldots, C_c$ in the GE decomposition, $s$ is the size of the separator set $S$, $r$ is the size of the set $R$ of remaining vertices, and edges exiting an odd component can only be connected to $S$. Every maximum matching matches all of $R$ to itself, matches each vertex of $S$ to a different odd component, and leaves exactly one unmatched vertex in every odd component that has no vertex matched to $S$. The value of a GE-decomposition is $c - s$, and an optimal GE-decomposition is one that maximizes  $c - s$. If $c > 0$ in a GE-decomposition, then necessarily $c > s$.

Let $H$ be a graph with $k$ vertices and $t$ edges for which the size of the maximum matching is $M(k,t)$. We make a {\em structural claim} that (at least) one of the following holds:

\begin{enumerate}

\item The optimal GE-decomposition for $H$ is a clique plus a set of isolated vertices. If the clique size is even it serves as $R$ in the GE-decomposition, whereas if its size is odd it serves as an odd component. In any case, every isolated vertex serves as an odd component, and $S$ is empty.

\item The optimal GE-decomposition for $H$ is a (complete) split graph, namely, a clique that serves as $S$, odd components that are singleton vertices, and edges between each odd component and all of $S$. (Note that for the split graph to be a GE-decomposition it must be that the independent set is larger than the clique, because $c - s$ needs to be positive.)

\item $M(k,t) = M(k,t+1)$.

\end{enumerate}

To prove the claim, we need to show that if the optimal GE-decomposition for $H$ is neither a clique nor a split graph, then we can find a graph with $t+1$ edges that has a GE-decomposition with the same value as that for $H$.

W.l.o.g.\ we may assume that $R \cup S$ is a clique, as otherwise we can add an edge to $R \cup S$ without decreasing $c - s$. Thereafter, if there are no odd components the GE-decomposition is a clique and we are done. Hence we may assume that there are odd components and $c > s$. As with the argument that $R \cup S$ is a clique, if we cannot add edges to $H$ while preserving $c - s$, it must be that each odd component is a clique, and every vertex in every odd component is connected to every vertex in $S$ (if $S$ exists). Moreover, $R$ is necessarily empty, as otherwise we can merge  it with an odd component (the component remains odd because $R$ is necessarily even), and completing that odd component to a clique we gain edges without changing $c-s$.

There remain several cases to consider:

\begin{enumerate}

\item There is exactly one odd component, $C_1$. Then $S$ is empty, and $R$ was already previously assumed to be empty. Hence $H$ is just the odd clique $C_1$, and so is its GE-decomposition.

\item There are (at least) two odd components that are not singletons, say $C_1$ and $C_2$. Then move all but one vertex from $C_1$ to $C_2$ (and remove the edges created between $C_1$ and $C_2$, and make $C_2$ into a clique). The number of edges increases without decreasing $c - s$.

\item $S$ is empty. From the previous two cases we may assume that there are at least two odd components, and at most one of them (say $C_1$) is a clique. As $R$ is empty as well, then the GE-decomposition is a clique ($C_1$).

\item $S$ is nonempty (and $R$ is empty, as argued above), there are at least $s + 1$ odd components, and exactly one odd component (say $C_1$) is not a singleton (if $C_1$ is a singleton then we have a split graph, as desired).

    \begin{enumerate}

    \item $|C_1| \ge c$. Merge $S$ and $s$ singletons into $C_1$, add edges to the new $C'_1$ to make it a clique, and disconnect the remaining singletons from $S$. The GE-decomposition becomes a clique $C'_1$, and the value $c-s$ does not change. The number of edges gained is $s |C_1|$, and lost is $(c-1-s)s$, so we strictly gained edges, as desired.

    \item $c > |C_1|$. Increase both $c$ and $s$ by one by making one vertex $v$ of $C_1$ a new singleton component $C'_1$, and moving one vertex $u$ from $C_1$ to $S$. After updating the edges, we lose $|C_1| - 2$ edges of $v$, but gain $c - 1$ new edges to $u$, so the number of edges increases.

    \end{enumerate}

 \end{enumerate}

Having proved our structural claim, it remains to show the quantitative bounds. For the lower bound, it suffices to check the size of the maximum matching in a clique and in a split graph, each with $\beta\binom{k}{2}$ edges, since the case that $M(k,t) = M(k,t+1)$ is handled by considering $t+1$, which only increases the computed bounds. For the upper bound it suffices to check the size of the maximum matching in graphs that are nearly a clique and nearly a split graph. 

In a clique of size $K$ the size of the maximum matching is $\left\lfloor \frac{K}{2} \right\rfloor$. Since the number of edges in the clique is $\binom{K}{2} = \beta \binom{k}{2}$, we must have $K \geq \sqrt{\beta} k$, which implies that the size of the maximum matching is at least $\left\lfloor \frac{\sqrt{\beta}}{2} k \right \rfloor$. 

For the matching upper bound, let $K$ be such that $\binom{K}{2} < \beta \binom{k}{2} \leq \binom{K+1}{2}$. The first inequality implies that $K < \sqrt{\beta}k+1$. Now define a graph that is a clique of size $K$ with an additional vertex that is connected to $\beta \binom{k}{2} - \binom{K}{2}$ vertices of the clique, and $k-(K+1)$ singleton vertices. The size of the maximum matching in this graph is $\left\lfloor \frac{K+1}{2} \right\rfloor$ and by the inequality above we have that $\left\lfloor \frac{K+1}{2} \right\rfloor \leq \left\lfloor \frac{\sqrt{\beta}}{2} k \right \rfloor + 1$. 

In a (complete) split graph, let $c$ denote the size of the independent set. The number of non-edges is $\binom{c}{2} = \left( 1 - \beta \right) \binom{k}{2}$ and hence $c \leq \left\lceil \sqrt{1-\beta} k \right\rceil$. The size of the maximum matching is $k - c$, since in the maximum matching every vertex of the clique is matched to a vertex in the independent set, and we have that 
$k - c \geq \left\lfloor \left( 1 - \sqrt{1-\beta} \right) k \right\rfloor$. 

For the matching upper bound, let $c$ be such that $\binom{c}{2} \leq \left( 1 - \beta \right) \binom{k}{2} < \binom{c+1}{2}$. The second inequality implies that $c \geq \left\lceil \sqrt{1-\beta} k \right\rceil - 1$. We define a graph on $k$ vertices as follows. Start with a split graph on $k$ vertices with an independent set of size $c$. Now take a vertex $v$ from the clique of the split graph and remove $\left( 1 - \beta \right) \binom{k}{2} - \binom{c}{2}$ edges connecting $v$ to vertices of the independent set. The resulting graph has $\beta \binom{k}{2}$ edges and its maximum matching has size $k-c$. By the inequality above we have that $k-c \leq \left\lfloor \left( 1 - \sqrt{1-\beta} \right) k \right\rfloor + 1$.
\end{proof}

\begin{proof}[of Theorem~\ref{thm:cover}] 
By Lemma~\ref{lem:covered} and a simple bound on binomial coefficients we obtain that 
\begin{equation*}\label{eq:N_bound}
N_{n,m,k,\beta} 
\leq \binom{m}{M(k,\beta)} \binom{n}{k - 2M(k, \beta)} 
\leq m^{M(k,\beta)} n^{k-2M(k,\beta)}.
\end{equation*}
Using Lemma~\ref{lem:matching} we thus obtain that 
\[
N_{n,m,k,\beta} 
\leq m^{\mu(k,\beta)+1} n^{k-2\mu(k,\beta) + 2},
\]
where recall the definition of $\mu(k,\beta)$ from~\eqref{eq:mu}. Theorem~\ref{thm:cover} now follows by observing that 
$\mu(k,\beta) = \left( 1-\sqrt{1-\beta} \right) k$ when $\beta \leq \frac{16}{25}$ and 
$\mu(k,\beta) = \frac{\sqrt{\beta}}{2} k$ 
when $\beta \geq \frac{16}{25}$. 

This bound is tight (up to factors of $k^{\Theta(k)}$) as can be seen by the following examples. 
First, let $\beta \leq \frac{16}{25}$. Suppose first that $m < kn$. Consider $m$ edges that form a split graph with a clique $K$ of size $(1 - \sqrt{1 - \beta})k$ joined to an independent set $I$ of size roughly $\frac{m}{k}$. Any choice of $\sqrt{1 - \beta}k$ vertices from $I$ completes together with $K$ a split graph with $\beta\binom{k}{2}$ edges. If $m > kn$, then make $K$ of size $m/n$ and $I$ of size $n - |K|$. Any choice of $\sqrt{1-\beta} k$ vertices from $I$ and $(1-\sqrt{1-\beta}) k$ vertices from $K$ forms a split graph with $\beta \binom{k}{2}$ edges.

Now let $\beta \geq \frac{16}{25}$. Consider $m$ edges that form a clique of size $\sqrt{2m}$. Any choice of $\sqrt{\beta}k$ vertices from the clique and $(1 - \sqrt{\beta})k$ vertices from the rest of the graph gives a subgraph with $\beta\binom{k}{2}$ edges.
\end{proof} 


\subsection{Explicit upper bounds for $\delta=1$}\label{sec:explicit}

The expression~\eqref{eq:upper_bound} in Theorem~\ref{thm:main_quantitative} gives an optimization problem to compute an upper bound on $\alpha_{\star} \left( \delta, \ell \right)$. As mentioned in the proof of Theorem~\ref{thm:main_quantitative}, for $\ell = 1$ this gives $\alpha_{\star} \left( \delta, 1 \right) \leq \delta$, which is tight. Here we investigate the optimization problem of~\eqref{eq:upper_bound} for other values of $\ell$. In particular, we solve this optimization problem for every $\ell$ when $\delta = 1$.  

For $\ell = 2$, the expression~\eqref{eq:upper_bound} gives 
\[
\alpha_{\star} \left( \delta, 2 \right) 
\leq \min_{\beta \in \left[ 0, 1 \right]} \max \left\{ 2 f \left( \beta, \delta \right) , \frac{\delta}{1-\beta} \right\}.
\]
For the two expressions in the display above to be less than $2$, we must have $\delta / (1-\beta) < 2$, or equivalently, $\beta < 1 - \delta / 2 \leq 1/2$. Hence recalling the definition of $f$ (see~\eqref{eq:f}) we obtain that 
\[
\alpha_{\star} \left( \delta, 2 \right) 
\leq \min_{\beta \in \left[ 0, \frac{1}{2} \right]} \max \left\{ 2 \left( \left( 2 - \delta \right) \sqrt{1-\beta} + \delta - 1 \right), \frac{\delta}{1-\beta} \right\}.
\]
Observe that the first expression in the display above is decreasing in $\beta$, while the second expression is increasing in $\beta$. Therefore the unique minimizer $\beta' = \beta' \left( \delta \right)$ satisfies 
\[
2 \left( \left( 2 - \delta \right) \sqrt{1-\beta'} + \delta - 1 \right) = \frac{\delta}{1 - \beta'}
\]
and we have that 
\[
\alpha_{\star} \left( \delta, 2 \right) \leq \frac{\delta}{1-\beta'}.
\]
For $\delta = 1$ we obtain that $\beta'(1) = 1 - 2^{-2/3}$ and thus $\alpha_{\star} \left( 1, 2 \right) \leq 2^{2/3} < 1.588$. 

Many of the observations made for the case of $\ell = 2$ above also apply for $\ell \geq 3$. We may write the expression~\eqref{eq:upper_bound} as follows: 
\[
\alpha_{\star} \left( \delta, \ell \right) 
\leq \min_{\beta \in \Delta_{\ell-1}} \max \left\{ \max_{i \in \left\{ 1, \ldots, \ell-1 \right\}} \frac{2 f \left( s_{i}, \delta \right)}{1-s_{i-1}}, \frac{\delta}{\beta_{\ell}} \right\},
\]
where $s_{i} = \sum_{j=1}^{i} \beta_{j}$ as before. If $\beta' = \beta' \left( \delta \right)$ denotes the optimizer then we must have $\beta_{\ell}' > \delta / 2 \geq 1/2$. This implies that $s_{i}' \leq s_{\ell-1}' \leq 1/2$ for every $i \leq \ell-1$, where $s_{i}' = \sum_{j=1}^{i} \beta_{j}'$. Hence recalling the definition of $f$ again we obtain that 
\[
\alpha_{\star} \left( \delta, \ell \right) 
\leq \min_{\beta \in \Delta_{\ell-1}, \beta_{\ell} \geq \frac{\delta}{2}} \max \left\{ \max_{i \in \left\{ 1, \ldots, \ell-1 \right\}} \frac{2 \left( \left( 2 - \delta \right) \sqrt{1-s_{i}} + \delta - 1 \right)}{1-s_{i-1}}, \frac{\delta}{\beta_{\ell}} \right\}. 
\]
This simplifies when $\delta = 1$: 
\[
\alpha_{\star} \left( 1, \ell \right) 
\leq \min_{\beta \in \Delta_{\ell-1}, \beta_{\ell} \geq \frac{1}{2}} \max \left\{ \max_{i \in \left\{ 1, \ldots, \ell-1 \right\}} \frac{2 \sqrt{1-s_{i}}}{1-s_{i-1}}, \frac{1}{\beta_{\ell}} \right\}. 
\] 
The optimizer $\beta' = \beta' (1)$ is such that all $\ell$ expressions on the right hand side of the display above are equal. That is, we must have 
\[
\frac{2 \sqrt{1-s_{i}'}}{1-s_{i-1}'} = \frac{1}{\beta_{\ell}'}
\]
for every $i \in \left\{ 1, \ldots, \ell - 1 \right\}$. This set of equations can be solved and we obtain the following optimizer: 
\begin{align*}
\beta_{1}' &= 1 - 2^{-\frac{2}{2^{\ell}-1}}, \\
\beta_{i}' &= 2^{-\frac{2^{i}-2}{2^{\ell}-1}} - 2^{-\frac{2^{i+1}-2}{2^{\ell}-1}} \qquad \text{ for } i \in \left\{ 2, \ldots, \ell -1 \right\}, \\
\beta_{\ell}' &= 2^{-\frac{2^{\ell}-2}{2^{\ell}-1}}. 
\end{align*}
The optimum is therefore 
$1/\beta_{\ell}' = 2^{\frac{2^{\ell}-2}{2^{\ell}-1}}$. In conclusion, we have proved the following corollary of Theorem~\ref{thm:main_quantitative}. 

\begin{cor}\label{cor:explicit_linear_probes}
We have that 
\[
\alpha_{\star} \left( 1, \ell \right) \leq 2^{1-\frac{1}{2^{\ell}-1}}.
\]
\end{cor}
Numerically, the bound above gives the following for small $\ell$: 
$\alpha_{\star} \left( 1, 2 \right) \leq 2^{2/3} < 1.588$, 
$\alpha_{\star} \left( 1, 3 \right) \leq 2^{6/7} < 1.812$, 
$\alpha_{\star} \left( 1, 4 \right) \leq 2^{14/15} < 1.910$, 
and $\alpha_{\star} \left( 1, 5 \right) \leq 2^{30/31} < 1.956$. 
These upper bounds should be compared to the lower bounds $\alpha_{\star} \left( 1, 2 \right) \geq 4/3$ and $\alpha_{\star} \left( 1, \ell \right) \geq \alpha_{\star} \left( 1, 3 \right) \geq 3/2$ for $\ell \geq 3$ (see Section~\ref{sec:algorithms}). 

\section*{Acknowledgements}

The problem considered here was proposed by David Gamarnik at the American Institute of Mathematics workshop ``{\em Phase transitions in randomized computational problems}" in June 2017~\cite{aimpl}. It arose from a discussion with Madhu Sudan, whose contribution to the inception of the problem is gratefully acknowledged. We thank AIM and the organizers, Amir Dembo, Jian Ding, and Nike Sun, for putting together the workshop. We also thank Jane Gao for initial discussions.


\bibliographystyle{abbrv}
\bibliography{bib}

\end{document}